\documentclass{article}
\usepackage{amsmath,amssymb,amsthm, tikz-cd, bbm}
\usepackage[shortlabels]{enumitem}
\usepackage{graphicx}
\usepackage[margin=1in]{geometry}

\theoremstyle{theorem}
\newtheorem{theorem}{Theorem}[section]
\newtheorem*{theorem*}{Theorem}
\newtheorem{corollary}[theorem]{Corollary}
\newtheorem*{corollary*}{Corollary}
\newtheorem{lemma}[theorem]{Lemma}
\newtheorem{proposition}[theorem]{Proposition}
\newtheorem*{thmdef*}{Theorem/Definition}

\theoremstyle{definition}

\newtheorem{example}[theorem]{Example}
\newtheorem{definition}[theorem]{Definition}
\newtheorem*{definition*}{Definition}

\newcommand{\C}{\mathbb{C}}
\newcommand{\Z}{\mathbb{Z}}

\newcommand{\cC}{\mathcal{C}}

\newcommand{\sq}{\subseteq}
\newcommand{\ds}{\dots}
\newcommand{\cds}{\cdots}

\newcommand{\F}{\mathbb{F}}
\newcommand{\Q}{\mathbb{Q}}
\newcommand{\N}{\mathbb{N}}

\newcommand{\fa}{\mathfrak{a}}

\DeclareMathOperator{\Hom}{Hom}
\DeclareMathOperator{\End}{End}

\DeclareMathOperator{\Mod}{Mod}

\newcommand{\Mustata}{Musta\c{t}\u{a}}

\setlist[enumerate]{itemsep=2pt, topsep=2pt, itemindent=10pt, label=(\roman*)}

\begin{document}
 	\title{Bernstein-Sato roots for monomial ideals in positive characteristic}
 	\author{Eamon Quinlan-Gallego \footnote{Partially supported by NSF grant DMS-1801697 and by the Ito Foundation for International Education Exchange.}}
 	
 	\maketitle
 	
 	\begin{abstract}
 		Following work of \Mustata \ and Bitoun we recently developed a notion of Bernstein-Sato roots for arbitrary ideals, which is a prime characteristic analogue for the roots of the Bernstein-Sato polynomial. Here we prove that for monomial ideals the roots of the Bernstein-Sato polynomial (over $\C$) agree with the Bernstein-Sato roots of the mod-$p$ reductions of the ideal for $p$ large enough. We regard this as evidence that the characteristic-$p$ notion of Bernstein-Sato root is reasonable.
 	\end{abstract}
 	
\section{Introduction}
Let $R = \C[x_1, \ds, x_n]$ be a polynomial ring over $\C$. We denote by $D_R$ the ring of $\C$-linear differential operators on $R$, i.e. the ring generated by $R$ and its derivations inside of $\End_\C(R)$. Let $f \in R$ be a nonzero polynomial. Bernstein \cite{Ber} and Sato \cite{SatoM} independently, and in different contexts, discovered the following fact: there is a nonzero polynomial $b(s) \in \C[s]$ and a differential operator $P(s) \in D_R[s]$ satisfying the following functional equation:
$$P(s) \cdot f^{s+1} = b(s) f^s.$$
The monic polynomial $b_f(s)$ of least degree for which there is some $P(s) \in D_R[s]$ satisfying the above equation is called the Bernstein-Sato polynomial for $f$. By a theorem of Kashiwara, it is known to have negative rational roots \cite{Kas76}. 

Since its inception the Bernstein-Sato polynomial has seen a wide variety of applications. In \cite{Mal74} Malgrange exhibited a relation between the roots of $b_f(s)$ and the eigenvalues of the monodromy action on the cohomology of the Milnor fibre of $f$. Kashiwara \cite{Kas83} and Malgrange \cite{Mal83} also used the existence of Bernstein-Sato polynomials to define $V$-filtrations with the purpose of defining nearby and vanishing cycles at the level of $D$-modules. Coming full circle, Budur, \Mustata \ and Saito then used this theory of $V$-filtrations to define the Bernstein-Sato polynomial $b_\fa(s)$ of an arbitrary ideal $\fa \sq R$, which still has negative rational roots in this setting.

A key application of Bernstein-Sato polynomials comes from the fact that the log-canonical threshold of $\fa$ (an invariant originally coming from complex analysis, but now with strong applications in birational geometry) is the smallest root of $b_\fa(-s)$. Moreover, any jumping number for the multiplier ideal in the interval $[\alpha, \alpha + 1)$, where $\alpha$ is the log-canonical threshold of $\fa$, is a root of the Bernstein-Sato polynomial \cite{BMSa}.

The test ideals, objects originally coming from the theory of tight closure \cite{HH90} \cite{HY03}, are known to give good characteristic $p$ analogues to multiplier ideals. It is thus reasonable to ask whether one could develop a theory of Bernstein-Sato polynomials in characteristic $p > 0$. This hope is encouraged by the fact that, in \cite{MTW}, the Bernstein-Sato polynomial of an ideal $\fa$ in characteristic zero has also been linked to certain characteristic $p$-invariants of a mod-$p$ reduction of $\fa$.

In \cite{Mustata2009} \Mustata \ was the first to explore this avenue of research for the case of a principal ideal $\fa = (f)$ in a regular $F$-finite ring of characteristic $p>0$. This technique has since then been refined by Bitoun in \cite{Bitoun2018}, and has also been generalized to the settings of unit $F$-modules \cite{Stad12} and $F$-regular Cartier modules \cite{BliStab16}.

In \cite{QG19} the approaches of \Mustata \ and Bitoun were expanded to arbitrary ideals $\fa \sq R$ and, in particular, a notion of Bernstein-Sato root of $\fa$ is defined by generalizing a previous definition of Bitoun. These Bernstein-Sato roots are characteristic-$p$ analogues of the roots of the Bernstein-Sato polynomial (it is a question in \cite{QG19} whether one can find an analogue for the multiplicity of a root).

The Bernstein-Sato roots of $\fa$ are negative, rational (that is, they lie in $\Z_{(p)}$) and encode some information about the $F$-jumping numbers of $\fa$ \cite{QG19}. Furthermore, the definition of Bernstein-Sato root in prime characteristic is compatible with that of the Bernstein-Sato polynomial in characteristic zero \cite[\S 6.1]{QG19}. 

Despite these nice properties about Bernstein-Sato roots in prime characteristic, if the concept is to be reasonable one would expect that if $\fa \sq \Z[x_1, \ds, x_n]$ is a monomial ideal then the Bernstein-Sato roots of the ideal $\fa_p$ in $\F_p[x_1, \ds, x_n]$ given by the image on $\fa$ in $\F_p[x_1, \ds, x_n]$ should recover those of the ideal $\fa_\C \sq \C[x_1, \ds, x_n]$, the expansion of $\fa$ to $\C[x_1, \ds, x_n]$. Indeed, we expect a similar statement for families of ideals in polynomial rings whose behavior does not depend on the characteristic of the base field. 

In this paper our goal is to show that this expectation is indeed true. More precisely, our theorem is as follows. 

\begin{theorem*} [\ref{thm-BSroot-mon-ideal}]
	Let $\fa \sq \Z[x_1, \ds, x_n]$ be a monomial ideal. Then the set of roots of $b_{\fa_\C}(s)$ coincides with the set of Bernstein-Sato roots of $\fa_p$ for $p$ large enough.
\end{theorem*}

Our proof relies heavily on results from \cite{BMSa06b}. In Section \ref{scn-background} we review the notion of Bernstein-Sato root as defined in \cite{QG19} as well as the needed theorems from \cite{BMSa06b}. We then prove our result in Section \ref{scn-main}. We finish with two examples in Section 4 that illustrate the behavior in small characteristics.

Let us fix the notation already used above: if $I \sq \Z[x_1, \ds, x_n]$ is an ideal we denote by $I_p$ the image of $I$ in $\F_p[x_1, \ds, x_n]$ and by $I_\C$ the expansion of $I$ to $\C[x_1, \ds, x_n]$. A ring $R$ of prime characteristic $p>0$ is $F$-finite if it is finite as a module over its subring $R^p$ of $p$-th powers.

\subsection*{Acknowledgements}

I would like to thank Karen Smith and Shunsuke Takagi for their encouragement and their guidance. I am especially grateful to Shunsuke Takagi for suggesting this problem to me. I would also like to thank the referee for their careful reading and suggestions. 

\section{Background} \label{scn-background}

Until stated otherwise we work with the following setup: $R$ is a regular ring of characteristic $p>0$ which is $F$-finite.

\subsection{Cartier operators}

We denote by $F: R \to R$ the Frobenius endomorphism on $R$ and, given an integer $e > 0$, we write $F^e$ for its $e$-th iterate. We define $F^e_*: \Mod(R) \to \Mod(R)$ to be the functor that restricts scalars via $F^e$. The $R$-module $F^e_* R$ is then equal to $R$ as an abelian group and we will denote an element $r \in R$ as $F^e_* r$ when viewed as an element of $F^e_* R$. In this way, the $R$-module action on $F^e_* R$ is given by $s \cdot F^e_* r = F^e_* (s^{p^e} r)$ for all $s, r \in R$. 

Given an ideal $I \sq R$ and and an integer $e > 0$, the ideal $I^{[p^e]}$ is defined to be the ideal generated by $p^e$-th powers of elements of $I$; that is, $I^{[p^e]} :=(f^{p^e}: f \in I)$. 

Given an integer $e> 0$ we let $\cC^e_R := \Hom_{R}(F^e_*R, R)$. An operator $\phi \in \cC^e_R$ acts on $R$ via $\phi \cdot r := \phi(F^e_* r)$ for all $r \in R$. In this way, given an ideal $I \sq R$ the new ideal $\cC^e_R \cdot I$ is generated by the set $\{\phi(F^e_*r ) : \phi \in \cC^e_R, r \in I\}$. When $R$ is a polynomial ring over $k$ and $I$ is principal the ideal $\cC^e_R \cdot I$ also admits the following description.
\begin{proposition} [{\cite[Prop. 2.5]{BMSm2008}}] \label{prop-Cef-repn}
	Let $R := k[x_1, \ds, x_n]$ be a polynomial ring over $k$, fix $e > 0$ and consider the set of multi-exponents $L := \{0, 1, \ds, p^e-1\}^n$. If $f$ is expressed in the $R^{p^e}$-basis $\{x^\gamma : \gamma \in L\}$ as $f = \sum_{\gamma \in L} g_\gamma^{p^e} x^\gamma$ then $\cC^e_R \cdot f = (g_\gamma : \gamma \in L)$. 
\end{proposition}

\subsection{The $\nu$-invariants} \label{subsn-nu-invt}
Let $\fa \sq R$ be an ideal. The invariants $\nu^J_\fa(p^e)$ were introduced in \cite{MTW}. We recall the definition.
\begin{definition} \label{def-nu-invt}
	Given a proper ideal $J \sq R$ containing $\fa$ in its radical and an integer $e > 0$ we define $\nu^J_\fa(p^e) := \max\{n \geq 0 : \fa^n \not\sq J^{[p^e]}\}.$ The set $\nu^\bullet_\fa(p^e) := \{ \nu^J_\fa(p^e) \ \big| \ (1) \neq \sqrt{J} \supseteq \fa \}$ is called the set of $\nu$-invariants of level $e$ for $\fa$. 
\end{definition}
It is clear from the definition that $\nu_\fa^{J^[p]}(p^e) = \nu^J_\fa(p^{e+1})$, and therefore the $\nu$-invariants come in a descending chain
$$\nu^\bullet_\fa(p^0) \supseteq \nu^\bullet_\fa(p^1) \supseteq \nu^\bullet_\fa(p^2) \supseteq \cds .$$

We will need the following results about $\nu$-invariants, which are well-known to experts.
\begin{proposition}[{\cite[Prop. 4.2]{QG19}}] \label{nu-invt-trun-ti-prop}
	Fix an integer $e > 0$. The set of $\nu$-invariants of level $e$ for $\fa$ is given by
	$$\nu^\bullet_\fa(p^e) = \bigg\{n \geq 0 \ \big| \ \cC_R^e \cdot \fa^n \neq \cC_R^e \cdot \fa^{n+1} \bigg\}.$$
\end{proposition}
\begin{corollary}[{\cite[Cor. 4.3]{QG19}}] \label{nu-invt-dynamics-cor}
	If $n \geq r p^e$ is a $\nu$-invariant of level $e$ then so is $n - p^e$. 
\end{corollary}
We next state following fact from \cite{MTW}, which connects the Bernstein-Sato polynomial with these characteristic $p$ invariants of singularities. 
\begin{proposition} [{\cite[Prop. 3.11]{MTW}}] Let $\fa \sq (x_1, \ds, x_n) \Z[x_1, \ds, x_n]$ be an ideal. Then for every $p \gg 0$ and every ideal $J \sq (x_1, \ds, x_n) \Z[x_1, \ds, x_n]$ containing $\fa$ in its radical we have
	$$b_{\fa_\C} (\nu^{J_p}_{\fa_p}(p^e)) \equiv 0 \ \mod p$$
	for all $e > 0$.
\end{proposition}
This has the following interesting corollary, which suggests a way of trying to find roots of $b_{\fa_\C}(s)$.
\begin{corollary}[{\cite[Rmk. 3.13]{MTW}}] \label{cor-find-roots}
	Suppose that for some ideal $J \sq (x_1, \ds , x_n) \Z[x_1, \ds, x_n]$ there exists some integer $M$, and a polynomial $P(t) \in \Q[t]$ such that $\nu^{J_p}_{\fa_p}(p^e) = P(p^e)$ whenever $p^e \equiv 1 \ \mod M$. Then $P(0)$ is a root of $b_{\fa_\C}(s)$.
\end{corollary}
\begin{proof}
	By Dirichlet's theorem, there are infinitely many primes $p$ with $p \equiv 1 \mod M$. Therefore, $b_{\fa_\C}(P(0)) \equiv b_{\fa_\C}(P(p^e))  \equiv 0 \mod p$ for infinitely many primes $p$, and thus $b_{\fa_\C}(P(0)) = 0$. 
\end{proof}
\subsection{The $\nu$-invariants of monomial ideals}

Fix a nonzero monomial ideal $\fa \sq \Z[x_1, \ds, x_n]$. In this setting whenever $J$ is a also a monomial ideal one can define the invariant $\nu^J_\fa(s)$ (c.f. Definition \ref{def-nu-invt}), in a characteristic-free way. First of all, given a monomial ideal $J \sq \Z[x_1, \ds, x_n]$ and a positive integer $q$ (not necessarily a prime power) we define an ideal $J^{[q]}$ of $\Z[x_1, \ds, x_n]$ as follows:
$$J^{[q]} := ( \mu^q : \mu \in J \text{ a monomial} ).$$
If $J$ is a monomial ideal containing $\fa$ in its radical we define
$$\nu^J_\fa(q) := \max \{n \geq 0 : \fa^n \not\sq J^{[q]}\}.$$
Observe that both of these notations are compatible with reduction mod-$p$ in the appropriate sense. 

We now state two theorems from \cite{BMSa06b}, which roughly say that the method suggested by Corollary \ref{cor-find-roots} for finding the roots of the Bernstein-Sato polynomial works for monomial ideals. While the behavior illustrated below has been shown to also hold for some examples of hypersurfaces \cite[Section 4]{MTW}, monomial ideals exhibit remarkable behavior in two ways: in order to recover all the roots it suffices to take $p^e \equiv 1 \mod M$ large and for $J$ to be a monomial ideal.

We state the theorems in a slightly weaker form which suffices for our purposes.

\begin{theorem}[{\cite[Thm. 4.1]{BMSa06b}}] \label{thm-BMS-mon-ideal-J}
	If $\fa \sq \Z[x_1, \ds, x_n]$ is a nonzero monomial ideal then there is a positive integer $M$ with the following property: if $J$ is a monomial ideal whose radical contains $\fa$ then there are rational numbers $\beta > 0$ and $\eta$ such that $\nu^J_\fa(q) = \beta q + \eta$ for all $q$ large enough with $q \equiv 1 \mod M$.
\end{theorem}

Observe that, by Corollary \ref{cor-find-roots}, the rational number $\eta$ in Theorem \ref{thm-BMS-mon-ideal-J} will be a root of $b_{\fa_\C}(s)$. 

\begin{theorem}[{\cite[Thm. 4.9]{BMSa06b}}] \label{thm-BMS-mon-ideal}
	Let $\fa \sq \Z[x_1, \ds, x_n]$ be a nonzero monomial ideal and $\alpha$ be a root of $b_{\fa_\C}(s)$. Then there is a monomial ideal $J$ together with a rational number $\beta$ and a positive integer $M$ such that $\nu^J_\fa(q) = \beta q + \alpha$ for $q$ large enough with $q \equiv 1  \mod M$. 
\end{theorem}
\subsection{Bernstein-Sato roots in positive characteristic}

We begin by reviewing the notion of Bernstein-Sato root from \cite{QG19}, to which we refer the reader for details. Let $R$ be a regular $F$-finite ring of prime characteristic $p>0$ and let $\fa \sq R$ be an ideal.

Using a choice of generators $\fa = (f_1, \ds, f_r)$ for $\fa$ one defines a directed system of modules $N^1 \to N^2 \to N^3 \to \cds$ and a family $s_{p^0}, s_{p^1}, s_{p^2}, \ds $ of differential operators on $R[t_1, \ds, t_r]$ with the following properties.
\begin{enumerate}
	\item The operators $s_{p^0}, s_{p^1}, \ds, s_{p^{e-1}}$ act on the module $N^e$ and the maps $N^e \to N^{e+1}$ are compatible with respect to this action.
	\item The operators $s_{p^i}$ are pairwise commuting, i.e. $s_{p^i} s_{p^j} = s_{p^j} s_{p^i}$ for all $i, j \geq 0$. 
	\item The operators $s_{p^i}$ satisfy $s_{p^i}^p = s_{p^i}$.
\end{enumerate}
Because we are in characteristic $p$, property (iii) is equivalent to $\prod_{j = 0}^{p-1} (s_{p^i} - j) = 0$. From properties (ii) and (iii) it follows that if an integer $e> 0$ is fixed then any module for the operators $s_{p^0}, s_{p^1}, \ds, s_{p^{e-1}}$ splits as a direct sum of multi-eigenspaces for these operators. In particular, we have
$$N^e = \bigoplus_{\alpha \in \F_p^e} N^e_\alpha$$
where, given $\alpha = (\alpha_0, \ds, \alpha_{e-1}) \in \F_p^e$ we define $N^e_\alpha := \{u \in N^e : s_{p^i} \cdot u = \alpha_i u \text{ for all } i=0,1,\ds, e-1\}$. 

Let $N = \varinjlim_e N^e$ be the limit of the directed system $N^e$. Since we have multi-eigenspace decompositions for each $N^e$ it is reasonable to ask whether $N$ has a multi-eigenspace decomposition -- although, in this case, it will be for infinitely many operators. The answer is positive and it leads to the notion of Bernstein-Sato root.

\begin{theorem}[{\cite[Prop. 6.1]{QG19}}] \label{thm-bs-root-decompn}
	We have a decomposition $N = \bigoplus_{\alpha \in \F_p^\N} N_\alpha$ where, given $\alpha = (\alpha_0, \alpha_1, \ds) \in \F_p^\N$, $N_\alpha = \{u \in N : s_{p^i} \cdot u = \alpha_i u \text{ for all } i \geq 0\}.$ Moreover, the number of $\alpha \in \F_p^\N$ for which $N_\alpha \neq 0$ is finite.
\end{theorem}

\begin{definition}[{\cite[Def. 6.2]{QG19}}]
	A $p$-adic integer $\alpha$ with $p$-adic expansion $\alpha = \alpha_0 + p \alpha_1 + p^2 \alpha_2 + \cds$ (i.e. $\alpha_i \in \{0, 1, \ds, p-1\}$) is a Bernstein-Sato root of $\fa$ if $N_{(\alpha_0, \alpha_1, \ds, )} \neq 0$. 
\end{definition}

Even though Bernstein-Sato roots are a-priori defined as $p$-adic integers, they turn out to be rational (i.e. they lie in the subring $\Z_{(p)}$ of $\Z_p$) and negative, and they are independent of the initial choice of generators for $\fa$. 

We end by stating the following characterization of Bernstein-Sato roots, which expresses them in terms of the $\nu$-invariants of $\fa$.

\begin{proposition} [{\cite[Prop. 6.13]{QG19}}]  \label{prop-new-charact}
	The following sets are equal.
	\begin{enumerate}[(a)]
		\item The set of Bernstein-Sato roots of the ideal $\fa$.
		\item The set of $p$-adic limits of sequences $(\nu_e) \sq \N$ where $\nu_e \in \nu^\bullet_\fa(p^e)$. 
		\item The set
		$$\bigcap_{e = 0}^\infty \overline{\nu^\bullet_\fa(p^e)},$$
		where $(\bar{ \ } )$ stands for $p$-adic closure.
	\end{enumerate}
\end{proposition}
\section{Main result} \label{scn-main}

Let $\fa \sq \Z[x_1, \ds, x_n]$ be a monomial ideal. One can then consider the expansion $\fa_\C $ of $\fa$ in the polynomial ring $\C[x_1, \ds, x_n]$ and let $b_{\fa_\C}(s)$ be its Bernstein-Sato polynomial. On the other hand, given a prime number $p$ we can also consider the ideal $\fa_p$, the image of $\fa$ in $\F_p[x_1, \ds, x_n]$ and consider its set of Bernstein-Sato roots (which, recall, lie in $\Z_{(p)}$).

In this section, we use results from \cite{BMSa06b} to show the following.

\begin{theorem} \label{thm-BSroot-mon-ideal}
	Let $\fa \sq \Z[x_1, \ds, x_n]$ be a monomial ideal. Then the set of roots of $b_{\fa_\C}(s)$ coincides with the set of Bernstein-Sato roots of $\fa_p$ for $p$ large enough.
\end{theorem}

We begin with a two preliminary results. The following lemma already appears implicitly in the proof of \cite[Prop. 3.2]{BMSm2008}.
\begin{lemma} \label{testideal-degree-lemma}
	Let $\fa$ be an ideal in the polynomial ring $R := \F_p[x_1, \ds, x_n]$ and let $e > 0$ be an integer. If $\fa$ can be generated by polynomials of degree at most $D$ then then $\cC^e_R \cdot \fa^m$ can be generated by polynomials of degree at most $\lfloor Dm/p^e \rfloor$. 
\end{lemma}
\begin{proof}
	First observe that $\fa^m$ is generated in degrees $\leq Dm$. That is, if we let $G := \fa^m \cap R_{\leq Dm}$ then $\fa^m = (f : f \in G)$. It follows that $\cC^e_R \cdot \fa^m = \sum_{f \in G} \cC^e_R \cdot f$ and therefore it suffices to show that if $f$ has degree $\leq Dm$ then $\cC^e_R \cdot f$ is generated by elements of degree $\leq Dm / p^e$. 
	
	Thus suppose $f$ has degree $\leq Dm$, and let $L$ be the set of multi-exponents $L := \{0, 1, \ds, p^e -1\}^n$. Suppose that, in the $R^{p^e}$-basis $\{x^\gamma : \gamma \in L\}$ for $R$, $f$ is expressed as $f = \sum_{\gamma \in L} g_\gamma^{p^e} x^\gamma$. Since $f$ has degree $\leq Dm$, all $g_\gamma$ have degrees $\leq Dm/p^e$. By Proposition \ref{prop-Cef-repn}, $\cC^e_R \cdot f = (g_\gamma : \gamma \in L)$ and the result follows.  
\end{proof}
\begin{lemma} \label{mu-J-lemma}
	Let $A$ be a commutative ring and consider the polynomial ring $R := A[x_1, \ds, x_n]$. Consider the monomial $\mu = x_1^{b_1} \cds x_n^{b_n}$ where $b_i \geq 0$ and the ideal $J = (x_1^{b_1 + 1}, \ds, x_n^{b_n + 1})$. Then for all monomial ideals $I \sq R$, $\mu \in I$ if and only if $I \not\sq J$. 
\end{lemma}
\begin{proof}
	The $(\Rightarrow)$ direction is clear, since $\mu \notin J$. For $(\Leftarrow)$, suppose $I \not \sq J$. This means that there exists some monomial $x_1^{a_1} \cds x_n^{a_n}$ in $I$ with $a_i \leq b_i$ for all $i$. By multiplying it with the appropriate monomial, we conclude $\mu \in I$. 
\end{proof}

We are now ready to prove a characteristic-$p$ analogue of Theorem \ref{thm-BMS-mon-ideal}, which will be key in the proof.

\begin{proposition} \label{prop-root-to-J-char-p}
	Let $\fa_p \sq \F_p[x_1, \ds, x_n]$ be a monomial ideal and suppose that $\alpha$ is a Bernstein-Sato root of $\fa_p$ and let $d > 0$ be an integer such that $\alpha(p^d - 1) \in \Z$. Then there is a monomial ideal $J$ whose radical contains $\fa$, a rational number $\beta$ and a sequence $e_i \nearrow \infty$ of positive integers such that
	$$\nu^J_{\fa_p}(p^{e_i d}) = \beta p^{e_i d} + \alpha.$$
\end{proposition}

We remark that, by \cite[Thm. 6.9]{QG19}, $\alpha$ is in $ \Z_{(p)}$ and thus we can always find some $d > 0$ such that $\alpha(p^d - 1) \in \Z$.

\begin{proof}
		By enlarging $d$ if necessary we may find $m \in \{0, 1, \ds, p^d-1\}$ and some rational number $\gamma$ with $-1 \leq \gamma \leq 0$ such that $\alpha = m + p^d \gamma$ (consider the $p$-adic expansion of $\alpha$, which is eventually repeating, c.f. \cite[\S 7]{QG19}). If $\alpha = \alpha_0 + p \alpha_1 + p^2 \alpha_2 + \cds$ is the $p$-adic expansion for $\alpha$ then, for all $e  >0$, 
		$$\alpha_0 + \cds + p^{ed -1}\alpha_{ed - 1} = \alpha - p^{ed}\gamma.$$
		
		By Proposition \ref{prop-new-charact}, $\alpha$ is the $p$-adic limit of a sequence $(\nu_e) \sq \N$ where $\nu_e \in \nu^\bullet_\fa(p^e)$. By passing to a subsequence we assume that $\nu_e \in \nu^\bullet_\fa(p^{ed})$ and that $\nu_e \equiv \alpha \mod p^{ed}$. By Corollary \ref{nu-invt-dynamics-cor} we can also assume that $0 \leq \nu_e < rp^{ed}$. From our assumptions it follows that for every $e>0$ there is some $s \in \{0, 1, \ds, r-1\}$ such that
		\begin{align*}
		\nu_e & = \alpha_0 + p \alpha_1 + \cds + p^{ed -1} \alpha_{ed-1} + p^{ed} s \\
			& = \alpha + (s - \gamma)p^{ed}.
		\end{align*}
		From Proposition \ref{nu-invt-trun-ti-prop} we conclude that for all $e>0$ there exists some $s \in \{0, 1 , \ds, r-1\}$ such that
		$$\cC^{ed}_R \cdot \fa^{\alpha + (s - \gamma)p^{ed}} \neq \cC^{ed}_R \cdot \fa^{\alpha + (s - \gamma)p^{ed} + 1}.$$
		Since $\{0, 1, \ds, r-1\}$ is a finite set, there exists some fixed $s_0 \in \{0, 1, \ds, r-1\}$ and a sequence $e_i \nearrow \infty$ such that
		$$\cC^{e_id}_R \cdot \fa^{\alpha + (s_0 - \gamma)p^{e_id}} \neq \cC^{e_id}_R \cdot \fa^{\alpha + (s_0 - \gamma)p^{e_id} + 1}.$$
		for all $i> 0$. 
		
		By Proposition \ref{prop-Cef-repn} the two ideals above are monomial ideals and, by Lemma \ref{testideal-degree-lemma}, there is some constant $K >0$ independent of $e$ such that both ideals are generated in degrees $\leq K$. As there are finitely many monomials of degree $\leq K$, by passing to a subsequence we may assume that there exists some monomial $\mu = x_1^{b_1} \cds x_n^{b_n}$ such that, for all $i > 0$, $\mu \in \cC^{e_i d}_R \cdot \fa^{\alpha + (s_0 - \gamma) p^{e_i d}}$ and $\mu \notin \cC^{e_i d}_R \cdot \fa^{\alpha + (s_0- \gamma)p^{e_i d} + 1}$. Finally, we let $J = (x_1^{b_1 + 1}, \ds , x_n^{b_n + 1})$ and, from Lemma \ref{mu-J-lemma}, we conclude that $\nu^J_\fa(p^{e_i d}) = \alpha + (s_0 - \gamma) p^{e_i d}$ as required.
\end{proof}

Before going into the proof of Theorem \ref{thm-BSroot-mon-ideal} we give an example to illustrate how one obtains a Bernstein-Sato root of $\fa_p$ from a root of $b_{\fa_\C}(s)$. 

\begin{example}
	Let $\fa = (X^2 Y Z, X Y^2 Z, X Y Z^2)$ (c.f. \cite[Ex. 5.2]{BMSa06b}). Then $b_{\fa_\C}(s) = (s + 3/4)(s + 5/4)(s + 6/4)(s + 1)^3$. Let us consider the root $\lambda = -5/4$. Theorem \ref{thm-BMS-mon-ideal} implies that there is a monomial ideal $J$, a positive integer $M$ and a rational number $\beta$ such that $\nu^J_\fa(q) = \beta q - 5/4$ whenever $q \equiv 1 \mod M$. 
	
	In this case, we claim that $J = (X^3, Y^3, Z^3)$ with $M = 4$ works. Indeed, the ideal $\fa^s$ is generated by monomials
	$$(X^2 Y Z)^u (X Y^2 Z)^v (X Y Z^2)^w = X^{2u + v + w} Y^{u + 2v + w} Z^{u + v + 2w}$$
	where $u,v,w$ range through all nonnegative integers with $u + v + w = s$, whereas $J^{[q]} = (X^{3q}, Y^{3q}, Z^{3q})$. We conclude that 
	$$\nu^J_\fa(q) := \max \{u + v + w \ | \ 2 u + v + w \leq 3q -1 \text{ and }  u + 2v + w \leq 3q -1 \text{ and } u + v + 2w \leq 3q -1\}.$$
	We claim that if $q \equiv 1 \mod 4$ then
	$$\nu^J_\fa(q) = \frac{9 q - 5}{4}.$$
	Indeed, adding the inequalities gives $\nu^J_\fa(q) \leq \lfloor (9q - 3)/4 \rfloor = (9 q - 5)/4$, and equality is proven by taking $u = v = (3 q - 3)/4$, $w = (3 q + 1)/4$.
	
	Now suppose that $p \equiv 3 \mod 4$. Then for all $e$ we have $p^{2e} \equiv 1 \mod 4$ and therefore $(9 p^{2e} - 5)/4 \in \nu^\bullet_\fa(p^{2e})$. Since the $p$-adic limit of the sequence $((9p^{2e} - 5)/4)_{ e = 0}^\infty$ is $-5/4$, Proposition \ref{prop-new-charact} implies that $-5/4$ is a Bernstein-Sato root of $\fa_p$, as required. The case $p \equiv 1 \mod 4$ follows similarly.
\end{example}

We are now ready to begin the proof of Theorem \ref{thm-BSroot-mon-ideal}.

\begin{proof}[Proof of Theorem \ref{thm-BSroot-mon-ideal}]
	First, let $\alpha$ be a root of the $b_{\fa_\C}(s)$. By Theorem \ref{thm-BMS-mon-ideal} we may find a monomial ideal $J \sq \Z[x_1, \ds, x_n]$, a rational number $\beta \in \Q$ and an integer $M$ such that $\nu^J_\fa(q) = \beta q + \alpha$ whenever $q$ is large enough and $q \equiv 1 \mod M$. Observe that, by replacing $M$ with a big multiple, $M$ can be chosen independently of $\alpha$, and we may also assume that $M \beta \in \N$. Let $p$ be a prime number that does not divide $M$ and such that $\alpha \in \Z_{(p)}$. Then there exists some $d$ such that $p^d \equiv 1 \mod M$ and therefore $\nu^J_\fa(p^{ed}) = \beta p^{ed} + \alpha$ for all $e > 0$. Since the $p$-adic limit of the sequence $(p^{ed} + \alpha)_{e = 0}^\infty$ is $\alpha$, Proposition \ref{prop-new-charact} implies that $\alpha$ is a Bernstein-Sato root of $\fa_p$.
	
	We now prove the other containment. We let $M$ be a number satisfying the conclusion of Theorem \ref{thm-BMS-mon-ideal-J} for the ideal $\fa$, and pick $p$ large enough so that it does not divide $M$. Suppose then that $\alpha$ is a Bernstein-Sato root of $\fa_p$, and we will show that $\alpha$ is a root of $b_{\fa_\C}$.
	
	By \cite[Thm. 6.9]{QG19}, $\alpha$ is in $\Z_{(p)}$ and thus we may find some $d > 0$ such that $\alpha(p^d - 1) \in \N$. By replacing $d$ with a multiple, we may also assume that $p^d \equiv 1 \mod M$. By Proposition \ref{prop-root-to-J-char-p} we can find some monomial ideal $J$ containing $\fa$ in its radical, a rational number $\beta$ and a sequence $e_i \nearrow \infty$ such that $\nu^J_\fa(p^{e_i d}) = \beta p^{e_id} + \alpha$. On the other hand, Theorem \ref{thm-BMS-mon-ideal} says that there are some rational numbers $\beta'$ and $\eta$ such that $\nu^J_\fa(q) = \beta' q + \eta$ for all $q \equiv 1 \mod M$ large enough. We conclude that $\beta' = \beta$ and $\eta = \alpha$ and, by Corollary \ref{cor-find-roots}, $\alpha$ is a root of $b_{\fa_\C}(s)$. 
\end{proof}

\section{Examples in small characteristics}

To finish we would like to illustrate the behavior in small characteristics by computing some examples. Let us remark that both of the examples below exhibit the following behavior: the Bernstein-Sato roots of $\fa_p$ are always roots of $b_{\fa_\C}(s)$ and, moreover, they are precisely the roots that lie in $\Z_{(p)}$. We do not know any example where this is not the case. 

We begin by making some general observations from \cite{BMSa06b}. Let $R = k[x_1, \ds, x_n]$ be a polynomial ring over an $F$-finite field $k$ of characteristic $p>0$, let $f_j = \prod_i x_i^{a_{ij}}$ for $j = 1, \ds, r$ be monomials in $R$ and let $\fa = (f_1, \ds, f_r)$ be the monomial ideal they generate. Let $\ell_i(t)$ be the linear form $\ell_i(t) = \sum_j a_{ij} t_j$ on $\Z^r$, where $i = 1, 2, \ds, n$. With this notation, the ideal $\fa^s$ is generated by monomials
$$x_1^{\ell_1(\beta)} \cds x_n^{\ell_n(\beta)}$$
where $\beta = (\beta_1, \ds, \beta_r) \in \N_0^r$ ranges through all tuples satisfying $\sum_j \beta_j = s$.

Next, observe that all $\nu$-invariants $\nu \in \nu^\bullet_\fa(p^e)$ arise as $\nu = \nu^J_\fa(p^e)$ where $J$ is a monomial ideal of the form $J = (x_1^{a_1}, \ds, x_n^{a_n})$ (see the proof of Proposition \ref{prop-root-to-J-char-p}). 

For such an ideal $J = (x_1^{a_1}, \ds, x_n^{a_n})$ we further observe the following:
\begin{align*}
\nu^J_\fa(p^e) & = \max\{s > 0 : \fa^s \not \sq J^{[p^e]}\} \\
	& = \max_{\beta \in \N_0^r} \{\textstyle \sum_j \beta_j : \ell_i(\beta) \leq a_i p^e - 1 \text{ for all } i\},
\end{align*}

\paragraph{Example 1:} Consider the ideal $\fa = (x_1^2, x_2^3)$. In this case, using computational software \cite{M2Dmod}, we find:
$$b_{\fa_\C} (s) = (s + \frac{5}{6}) (s + \frac{7}{6}) (s + \frac{4}{3}) (s + \frac{3}{2}) ( s+ \frac{5}{3}) (s + 2),$$
For $J = (x_1^{a_1}, x_2^{a_2})$ we have $\ell_1(t_1, t_2) = 2 t_1$, $\ell_2(t_1, t_2) = 3 t_2$ and therefore
\begin{align*}
\nu^J_\fa(p^e) & = \max_{t_1, t_2 \in \N_0} \{t_1 + t_2 : 2t_1 \leq a_1 p^e - 1, 3t_2 \leq a_2 p^e - 1\} \\
	& = \lfloor \frac{p^e a_1 - 1}{2} \rfloor + \lfloor \frac{p^e a_2  - 1}{3} \rfloor
\end{align*}
and therefore
$$\nu^\bullet_\fa(p^e) = \bigg\{\lfloor \frac{p^e a_1 - 1}{2} \rfloor + \lfloor \frac{p^e a_2  - 1}{3} \rfloor : a_1, a_2 \in \N \bigg\}$$

 Suppose that $p = 2$ and that $e$ is even. Then for all $a_1 \in \N$ we have $\lfloor a_1p^{e} - 1)/2 \rfloor = a_1 p^{e-1} - 1$, while
		$$\lfloor \frac{a_2 p^e -1}{3} \rfloor = \begin{cases}
		c p^e - 1 \text{ if } a_2 = 3c \\
		(c - \frac{1}{3})p^e - \frac{1}{3} \text{ if } a_2 = 3c - 1 \\
		(c - \frac{2}{3})p^e - \frac{2}{3} \text{ if } a_2 = 3 c - 2,
		\end{cases}$$
		where we always take $c \in \N$. We conclude that, for even $e$,
		\begin{align*}
		\nu^\bullet_\fa(p^e) & = \bigg\{a_1 p^{e-1}- 2 : a_1, \in \N \bigg\} \cup \bigg\{a_1 p^{e-1}+ (c - \frac{1}{3}) p^e - \frac{4}{3}: a_1, c \in \N \bigg\} \\
		& \hspace*{50pt} \cup \bigg\{a_1 p^{e-1}+ (c - \frac{2}{3}) p^e - \frac{5}{3}: a_1, c \in \N \bigg\} 
		\end{align*}
		and therefore $BS(\fa) = \{-4/3, -5/3, -2\}$ by Proposition \ref{prop-new-charact}.
		
 When $p = 3$ a similar computation yields
		$$\nu^\bullet_\fa(p^e) = \bigg\{ a_2p^{e-1} - 2 : a_2 \in \N \bigg\}  \cup \bigg\{ (c - \frac{1}{2})p^e + a_2 p^{e-1} - \frac{3}{2} : c, a_2 \in \N \bigg\}$$
		and therefore $BS(\fa) = \{-3/2, -2\}$.
		
		When $p \geq 5$ the same method yields $BS(\fa) = \{-5/6, -7/6, -4/3, -3/2, -5/3, -2\}$ as predicted by Theorem \ref{thm-BSroot-mon-ideal}

\paragraph{Example 2:} 

Let $\fa = (x_2 x_3, x_1 x_3, x_1 x_2)$. By again using \cite{M2Dmod} we find that
$$b_{\fa_\C} = (s + 2)^2 (s + \frac{3}{2}),$$
and for all $p > 2$ we obtain precisely the above Bernstein-Sato roots.

In this case we have $\ell_1(t_1, t_2, t_3) = t_2 + t_3$, $\ell_2(t_1, t_2, t_3) = t_1 + t_3$ and $\ell_3(t_1, t_2, t_3) = t_1 + t_2$. For $J = (x_1^{a_1}, x_2^{a_2}, x_3^{a_3})$ we claim
$$\nu^J_\fa(p^e) = \min \big\{p^e (a_1 + a_2) - 2, p^e(a_1 + a_3) - 2, p^e (a_2 + a_3) - 2, \lfloor \frac{ p^e(\sum_j a_j) - 3}{2} \rfloor \big\}.$$
Indeed, when the minimum is given by $p^e(a_1 + a_2) - 2$ then we have $a_1 + a_2 \leq a_3$ and we can take $t_1 = a_2 p^e - 1$, $t_2 = a_1 p^e - 1$. The case where the minimum is $p^e(a_1 + a_3) - 2$ and the case where the minimum is $p^e(a_2 + a_3) - 2$ follow similarly. We therefore may assume that the minimum is $\lfloor (p^e(a_1 + a_2 + a_3) - 3)/2 \rfloor$ and that $a_1 + a_2 > a_3$, $a_1 + a_3 > a_2$ and $a_2 + a_3 > a_1$. The case where $p^e(\sum_j a_j) - 3$ is divisible by 2 is dealt with by taking $t_i = \frac{1}{2}(p^e(\sum_j a_j - 2 a_i) - 1)$. In the case where $p^e(\sum_j a_j) - 3$ is not divisible by 2 we can take $t_1 = \frac{1}{2} (p^e(-a + b + c) + 2)$, $t_2 = \frac{1}{2} (p^e(a_1 - a_2 + a_3) - 2)$ and $t_3 = \frac{1}{2} (p^e(a_1 + a_2 - a_3) - 2)$. 
%
%TRICK TO FIND THE ABOVE: INVERT THE MATRIX OF EXPONENTS.
%

It follows that for $p = 2$ we have 
$$\nu^\bullet_\fa(p^e) = \{p^e a - 2: a \in \N_0\}$$
and therefore $BS(\fa) = \{-2\}$. 

For $p > 2$ we find that 
$$\nu^\bullet_\fa(p^e) = \{p^e a - 2: a \in \N_0\} \cup \{(p^e a - 3)/2 : a \in 2 \N  + 1\}$$
and therefore $BS(\fa) = \{-3/2, -2\}$, again in agreement with Theorem \ref{thm-BSroot-mon-ideal}.

\bibliography{biblio}
\bibliographystyle{alpha}

\end{document}